\documentclass{article}%
\usepackage{amsmath}
\usepackage{amsfonts}
\usepackage{amssymb,float}
\usepackage{latexsym}
\usepackage[latin1]{inputenc}
\usepackage{amsmath}
\usepackage{amssymb,float}
\usepackage{latexsym}
\usepackage{epstopdf}
\usepackage{geometry}
\usepackage[active]{srcltx}
\usepackage{graphicx}
\usepackage{epstopdf}
\usepackage{enumerate}
\usepackage[ bookmarks=true,         bookmarksnumbered=true, colorlinks=true,
pdfstartview=FitV, linkcolor=blue, citecolor=blue, urlcolor=blue]{hyperref}
\usepackage{amssymb}%
\providecommand{\U}[1]{\protect\rule{.1in}{.1in}}
\newtheorem{theorem}{Theorem}[section]

\newtheorem{lemma}{Lemma}[section]

\newtheorem{remark}{Remark}[section]

\newenvironment{proof}[0]{\paragraph{Proof.}}{\rule{0.5em}{0.5em}}
\begin{document}
\noindent This paper will appear in the special  number\\ dedicated to the memory of Murrey Rosenblatt  \\( Journal of Time Series Analysis).
\begin{center}
{\LARGE A Local Limit Theorem for Linear Random Fields}

\bigskip\bigskip

Timothy Fortune$^{a}$, Magda Peligrad$^{b}$ and Hailin Sang$^{c}$
\end{center}

\bigskip

$^{a}$ Department of Statistics, University of Connecticut, Storrs, CT
06269, USA \newline timothy.fortune@uconn.edu

$^{b}$ Department of Mathematical Sciences, University of Cincinnati,
Cincinnati, OH 45221, USA\newline E-mail address: peligrm@ucmail.uc.edu

$^{c}$ Department of Mathematics, University of Mississippi, University, MS
38677, USA \newline sang@olemiss.edu

\bigskip

\noindent\textbf{Keywords}: \noindent linear random fields, local limit
theorem, long memory.

\bigskip


\begin{center}
\bigskip

\textbf{Abstract}
\end{center}

In this paper, we establish a local limit theorem for linear fields of random
variables constructed from independent and identically distributed innovations
each with finite second moment. When the coefficients are absolutely summable
we do not restrict the region of summation. However, when the coefficients are
only square-summable we add the variables on unions of rectangle and we impose
regularity conditions on the coefficients depending on the number of
rectangles considered. Our results are new also for the dimension $1$, i.e.
for linear sequences of random variables. The examples include the
fractionally integrated processes for which the results of a simulation study
is also included.

\section{Introduction}

Consider independent and identically distributed (i.i.d.) standard normal
random variables $\{Z_{k}\}_{k=1}^{n}$ and their sum $S_{n}=\sum_{k=1}%
^{n}Z_{k}$. In this context, we can define a sequence of measures given by
\begin{equation}
\mu_{n}(a,b)=\sqrt{2\pi n}\;P(S_{n}\in(a,b))=\int_{a}^{b}e^{-\frac{1}{2n}
x^{2}}\;dx, \label{3}%
\end{equation}
and with this specific form, one can easily see that the integrand converges
to one as $n\rightarrow\infty$. This sequence of measures therefore converges
to Lebesgue measure. The result is also true for the situation when
$\{Z_{k}\}_{k=1}^{n}$ is merely a sequence of i.i.d. random variables
satisfying the central limit theorem (CLT). A result such as this is called a
local limit theorem. A local limit theorem is much more delicate than the
associated CLT.

Local limit theorems have been studied intensively for\ the case of lattice
random variables and the case of non-lattice random variables. The lattice
case means that there exists $v>0$ and $a\in\mathbb{R}$ such that the values
of $Z_{0}$ are concentrated on the lattice $\{a+kv:k\in\mathbb{Z}\}$, whereas
the non-lattice case means that no such $a$ and $v$ exists. In this paper, we
consider the non-lattice case.

For sequences of i.i.d. random variables, the local limit theorem in the
non-lattice case is due to Shepp (1964) and the case of i.i.d. random
vectors is considered by Stone (1965). We also refer the reader to the
books by Ibragimov and Linnik (1971), Petrov (1975), and
Gnedenko (1962). Some papers containing classes of independent non
identically distributed random variables include Mineka and Silverman
(1970), Shore (1978) and Maller (1978). For more recent
results we mention the paper by Dolgopyat (2016) and the references therein.

Linear random fields (also known in the statistical literature as spatial
linear processes) have been extensively studied in probability and statistics.
For example, Mallik and Woodroofe (2011) studied the CLT for linear
random fields, and Sang and Xiao (2018) established exact moderate
and large deviation asymptotics for linear random fields under moment or
regularly varying tail conditions by extending the methods for linear
processes in Peligrad et al. (2014). With a conjugate method,
Beknazaryan, Sang, and Xiao (2019) studied the Cramér type
moderate deviation for partial sums of linear random fields. We refer to Sang
and Xiao (2018) for a brief review of the study of asymptotic
properties of linear random fields and to Koul, Mimoto, and Surgailis
(2016), Lahiri and Robinson (2016) and the references therein for
recent developments in statistics. However, to the best of our knowledge, the
local limit results for linear random fields, or even for one dimensional
indexed linear processes, have not yet been established in the literature.

In this paper, we consider linear random fields of the form
\begin{equation}
X_{j}=\sum_{i\in{\mathbb{Z}}^{d}}a_{i}\varepsilon_{j-i} \label{deflin}%
\end{equation}
defined on ${\mathbb{Z}}^{d}$, where the innovations $\varepsilon_{i}$ are
i.i.d. random variables with mean zero $({{\mathbb{E}}\,}\varepsilon_{i}=0)$,
finite variance $({{\mathbb{E}}\,}\varepsilon_{i}^{2}=\sigma_{\varepsilon}%
^{2})$, and non-lattice distribution and where the collection $\{a_{i}%
:i\in{\mathbb{Z}}^{d}\}$ of real coefficients satisfies
\begin{equation}
\sum\nolimits_{i\in{\mathbb{Z}}^{d}}a_{i}^{2}<\infty. \label{cond coef}%
\end{equation}
As a matter of fact, the field $X_{j}$ given in (\ref{deflin}) exists in
$L^{2}({\mathbb{R}})$ and almost surely if and only if (\ref{cond coef}%
)\ is satisfied. We say that the process has \textit{long memory (long range
dependence)} if $\sum\limits_{i\in{\mathbb{Z}}^{d}}|a_{i}|=\infty$.

Let $\Gamma_{n}^{d}$ be a sequence of finite subsets of ${\mathbb{Z}}^{d}$,
and define the sum
\begin{equation}
S_{n}=\sum_{j\in\Gamma_{n}^{d}}X_{j} \label{S_n}%
\end{equation}
with variance
\begin{equation}
B_{n}^{2}=\mathrm{Var}(S_{n}). \label{B_n}%
\end{equation}
We may express (\ref{deflin}) as
\[
X_{j}=\sum_{i\in{\mathbb{Z}}^{d}}a_{j-i}\varepsilon_{i},
\]
from which it is easily apparent that
\[
\mathrm{var}{(}X_{j})=\sigma_{\varepsilon}^{2}\sum_{i\in{\mathbb{Z}}^{d}}%
a_{i}^{2}.
\]
The sum $S_{n}=\sum\limits_{j\in\Gamma_{n}^{d}}X_{j},$ expressed as an
infinite linear combination of the innovations, is given by
\begin{equation}
S_{n}=\sum\limits_{i\in{\mathbb{Z}}^{d}}b_{n,i}\;\varepsilon_{i}, \label{rfsn}%
\end{equation}
where
\[
b_{n,i}=\sum\limits_{j\in\Gamma_{n}^{d}}a_{j-i},
\]
and similar to our earlier observation,
\[
B_{n}^{2}=\mathrm{Var}(S_{n})=\sigma_{\varepsilon}^{2}\sum_{i\in{\mathbb{Z}%
}^{d}}b_{n,i}^{2}.
\]
Without loss of generality, throughout the paper we assume that
$\sigma_{\varepsilon}^{2}=1$. Note that, by the representation (\ref{rfsn}%
),\ $S_{n}$ can be expressed as a sum of independent variables. However, the
local limit theorems available for sums of independent random variables that
are not identically distributed involve rather strong degrees of stationarity
which are not satisfied by (\ref{rfsn}). Building on the previous work of Shore (1978), we are able to show that the local limit theorem holds for all
the situations including the long memory linear random fields, assuming
reasonable requirements of the innovations and of the sets $\Gamma_{n}^{d}$.

As a matter of fact, we shall establish the following uniform local limit
theorem: For all continuous complex-valued functions $h(x)$ with $|h|\in L^{1}%
(\mathbb{R})$ and with Fourier transform $\hat{h}$ real and with compact
support,
\begin{equation}
\lim_{n\rightarrow\infty}\sup_{u\in\mathbb{R}}\bigg|\sqrt{2\pi}B_{n}%
Eh(S_{n}-u)-[\exp(-u^{2}/2B_{n}^{2})]\int h(x)\lambda
(dx)\bigg|=0,\label{LCLTnonlat}%
\end{equation}
where $\lambda$ is the Lebesgue measure. Here we require that $B_{n}%
\rightarrow\infty$ as $n\rightarrow\infty$. By arguments in Section 4 of
Hafouta and Kifer  (2016) this result implies that (\ref{LCLTnonlat}%
)\ also holds for the class of real continuous functions with compact support
and by the Theorem 10.7 in Breiman (1992) it follows that
\[
\lim_{n\rightarrow\infty}\sup_{u\in\mathbb{R}}\bigg|\sqrt{2\pi}B_{n}P(a+u\leq
S_{n}\leq b+u)-[\exp(-u^{2}/2B_{n}^{2})](b-a)\bigg|=0,
\]
for any $a<b$. In particular, since $B_{n}\rightarrow\infty$ as $n\rightarrow
\infty$, then for fixed $A>0$,
\[
\lim_{n\rightarrow\infty}\sup_{|u|\leq A}\bigg|\sqrt{2\pi}B_{n}P(a+u\leq
S_{n}\leq b+u)-(b-a)\bigg|=0.
\]
If we further take $u=0$, then,
\[
\lim_{n\rightarrow\infty}\sqrt{2\pi}B_{n}P(S_{n}\in\lbrack a,b])=b-a.
\]
In other words, the sequence of measures $\sqrt{2\pi}B_{n}P(S_{n}\in\lbrack
a,b])$ of the interval $[a,b]$ converges to Lebesgue measure.

It should be noted that the local limit theorem, as formulated in
(\ref{LCLTnonlat}), is useful to the study of recurrence conditions for
$S_{n}$, as explained in Orey  (1966) and Mineka and Silverman
 (1970).

The paper is organized as follows. In Section 2 we state and comment on the
results, which include the long memory case. Section 3 is dedicated to
examples of long memory time series to which we can apply the local limit
theorem stated in the previous section. In Section 4 we summarize the result
of a simulation study, designed to analyze the performance of our asymptotic
local theorem for a finite sample. Finally, Section 5 contains the proof of
the main result.

A few remarks about notation and terms used in the paper follow. In
constructing the sum $S_{n}$ that we analyze in this paper, we make use of a
sequence $\Gamma_{n}^{d}$ of subsets of ${\mathbb{Z}}^{d}$. For use with the
long memory case, for each $n$, we will construct the sequence $\Gamma_{n}%
^{d}$ of sets using a union of rectangles, whose dimensions could depend on
$n$. For $\underline{n}(w)=\underline{n}(w,n)=(\underline{n}_{1}%
(w),\underline{n}_{2}(w),\cdots,\underline{n}_{d}(w))\in{\mathbb{Z}}^{d}$ and
$\overline{n}(w)=(\overline{n}_{1}(w),\overline{n}_{2}(w),\cdots,\overline
{n}_{d}(w))\in{\mathbb{Z}}^{d}$ with $\underline{n}(w)\leq\overline{n}(w)$,
where $1\leq w\leq J_{n}$, put $\Gamma_{n}^{d}(w)=\prod_{\ell=1}%
^{d}[\underline{n}_{\ell}(w),\overline{n}_{\ell}(w)]\cap{\mathbb{Z}}^{d}.$ Any
set of this form will be called a \textit{discrete rectangle}. In general, we
require the index sets to be of the form
\begin{equation}
\Gamma_{n}^{d}=\bigcup\limits_{w=1}^{J_{n}}\Gamma_{n}^{d}(w), \label{defgamma}%
\end{equation}
where $\{\Gamma_{n}^{d}(w)\}_{w=1}^{J_{n}}$ is a pairwise disjoint family of
discrete rectangles. Throughout the paper, we demand that $|\Gamma_{n}%
^{d}|\rightarrow\infty$ as $n\rightarrow\infty$. Here, for $\Gamma
\subset{\mathbb{Z}}^{d}$, we denote the cardinality of $\Gamma$ by $|\Gamma|$.
For $n=(n_{1},...n_{d})$ the Euclidian norm will be denoted by $\Vert
n\Vert=(n_{1}^{2}+n_{2}^{2}+...+n_{d}^{2})^{1/2}$. Let $\{a_{n}\}_{n=1}%
^{\infty}$ and $\{b_{n}\}_{n=1}^{\infty}$ be real-valued sequences. To
indicate relative growth rates at infinity, we use $a_{n}\propto b_{n}$ to
indicate that $a_{n}/b_{n}\rightarrow C\in{\mathbb{R}}^{+}$, and the
particular case when $C=1$ is denoted $a_{n}\;\mathtt{\sim~}b_{n}$. By
$a_{n}=o(b_{n})$ we understand that $a_{n}/b_{n}\rightarrow0$ and
$a_{n}=O(b_{n})$ means that $\lim\sup|a_{n}/b_{n}|<C$ for some positive
numbrer $C.$ Throughout the paper, an indicator function will be denoted as
$\mathcal{I}$. A function $l:[0,\infty)\rightarrow{\mathbb{R}}$ is referred to
as \textit{slowly varying} (at $\infty$) if it is positive and measurable on
$[A,\infty)$ for some $A\in{\mathbb{R}}^{+}$ such that $\lim
\limits_{x\rightarrow\infty}l(\lambda x)/l(x)=1$ holds for each $\lambda
\in{\mathbb{R}}^{+}$. The integer part of a real number $x$ will be denoted by
$\lfloor x\rfloor.$


\section{Main Results}

\label{main}

In this work, we investigate the conditions under which the local limit
theorem holds for the partial sums of the linear random fields given by
(\ref{deflin}). Before we can treat the local limit theorem of this paper, we
mention the following CLT for linear random fields which is a variant of
Corollary 2 and Corollary 4 of Mallik and Woodroofe  (2011). For $d=1$
and $J_{n}=1$ with $\Gamma_{n}^{1}=\{1,2,\cdots,n\}$ the result is Theorem
18.6.5 in Ibragimov and Linnik  (1971).

\begin{theorem}
\label{thm1} (Mallik and Woodroofe, 2011) Let $S_{n}$ and $B_{n}$ be
defined as in (\ref{S_n}) and (\ref{B_n}). Assume that $B_{n}\rightarrow
\infty.$ When the field has long range dependence we additionally require that
the sets $\Gamma_{n}^{d}$ are constructed as a disjoint union of $J_{n}$
discrete rectangles, where $J_{n}=o(B_{n}^{2})$, while otherwise no such
restriction is required. Under these conditions, $S_{n}/B_{n}$ converges in
distribution to the standard normal distribution.
\end{theorem}

\begin{remark}
In case $\sum\nolimits_{i\in{\mathbb{Z}}^{d}}|a_{i}|<\infty$, this theorem was
proved in Corollary 2 of Mallik and Woodroofe  (2011). When the field
has long range dependence the result of this theorem is a version of their
Corollary 4. Indeed, from relation (11) in the proof of Proposition 2 of the
same paper, the condition $\sup_{i\in\mathbb{Z}^{d}}|b_{n,i}|/B_{n}%
\rightarrow0$ is satisfied if $J_{n}=o(B_{n}^{2})$.
\end{remark}

\begin{remark}
If $J_{n}=1$, then $\Gamma_{n}^{d}$ consists of only one rectangle $\Gamma
_{n}^{d}(w)=\prod_{\ell=1}^{d}[\underline{n}_{\ell}(w),\overline{n}_{\ell
}(w)]\cap{\mathbb{Z}}^{d}.$ The condition $B_{n}\rightarrow\infty,$ implies
that $\max_{1\leq\ell\leq d}|\ \overline{n}_{\ell}(w)-\underline{n}_{\ell
}(w)|\rightarrow\infty$ as $n\rightarrow\infty.$ Note that if more than one
difference among $(\overline{n}_{\ell}(w)-\underline{n}_{\ell}(w))_{1\leq
\ell\leq d}$ tend to infinity, they can grow at independent rates.
\end{remark}

\begin{remark}
\label{r2} Given that $R_{0}$ is an open connected subset of $(-1/2,1/2]^{d}$
satisfying some regularity conditions and $\{\mu_{n}\}$ is a sequence of
positive numbers such that $\mu_{n}\rightarrow\infty$ as $n\rightarrow\infty$,
Lahiri and Robinson (2016) studied the central limit theorems for the
sums of linear random fields over dilated regions $\Gamma_{n}^{d}=R_{n}%
\cap{\mathbb{Z}}^{d}$, where $R_{n}=\mu_{n}R_{0}.$ In particular, when the
coefficients are of the form $a_{i}=l(||i||)/||i||^{\alpha}$ with
$d/2<\alpha<d$,
$l$ a slowly varying function at infinity, then, as shown in Lahiri and Robinson (2016), 
$B_{n}^{2}\propto\mu_{n}^{3d-2\alpha}l^{2}(\mu_{n})$. However, since the
volume of $R_{n}=O(\mu_{n}^{d})$, the sample size $|\Gamma_{n}^{d}|=O(\mu
_{n}^{d})$. We can separate $\Gamma_{n}^{d}$ into $J_{n}$ disjoint rectangles
with $J_{n}=O(\mu_{n}^{d})$. Since $B_{n}^{2}\propto\mu_{n}^{3d-2\alpha}%
l^{2}(\mu_{n})$ and $3d-2\alpha>d$, it is easy to see that $J_{n}=o(B_{n}%
^{2})$. Hence their central limit theorem (Theorem 3.2 there) in the long
memory case is a direct consequence of Theorem \ref{thm1} here.
\end{remark}

Denote the characteristic function of $\varepsilon_{0}$ by $\varphi
_{\varepsilon}(t):=\mathbb{E}(\exp\{it\varepsilon_{0}\}).$ It is well known
that $\varepsilon_{0}$ not having a lattice distribution is equivalent to
$|\varphi_{\varepsilon}(t)|<1$ for all $t\neq0.$ On the other hand, the Cramér
condition means that $\lim\sup_{|t|\rightarrow\infty}|\varphi_{\varepsilon
}(t)|<1$. Thanks to the Riemann-Lebesgue lemma, the Cramér condition is
automatically satisfied if the distribution function of $\varepsilon_{0}$ is
absolutely continuous with respect to the Lebesgue measure. It should be
mentioned that $\varepsilon_{0}$ has a non-lattice distribution whenever
$\varphi_{\varepsilon}(t)$ satisfies the Cramér condition. See Lemma
\ref{lemma1}.

$[$The \textquotedblleft Cramér condition" defined in the preceding paragraph
is different from, and has no particular connection with, another condition
(involving the existence of moment generating functions on certain domains)
that has absolutely no role in this paper but has elsewhere in the probability
theory literature sometimes been referred to as the \textquotedblleft Cramér
condition".$]$

\begin{theorem}
\label{thm2} Let $S_{n}$ and $B_{n}$ be defined as in (\ref{S_n}) and
(\ref{B_n}) and assume that $B_{n}\rightarrow\infty$. In the case
$\sum\nolimits_{i\in{\mathbb{Z}}^{d}}|a_{i}|<\infty$, we assume that
$\varepsilon_{0}$ is non-lattice. If the field has long range dependence, we
assume that the innovations satisfy the Cramér condition and that the sets
$\Gamma_{n}^{d}$ are constructed as a disjoint union of $J_{n}$ discrete
rectangles and we require that
\begin{equation}
\frac{J_{n}^{2/d}\log(B_{n})}{\sup_{i\in{\mathbb{Z}}^{d}}|b_{n,i}|^{2/d}}\rightarrow
0\text{ as }n\rightarrow\infty. \label{condTh2}%
\end{equation}
Under these conditions, (\ref{LCLTnonlat}) holds.
\end{theorem}

\begin{remark}
Because $(X_{k})$ is stationary we always have
\begin{equation}
var(S_{n})=B_{n}^{2}\leq|\Gamma_{n}^{d}|^{2}\mathbb{E}(X_{0}^{2}).
\label{bound B}%
\end{equation}
So, if $J_{n}^{2/d}\log|\Gamma_{n}^{d}|/\sup_{i\in{\mathbb{Z}}^{d}}|b_{n,i}%
|^{2/d}\rightarrow0,$ then (\ref{condTh2}) is satisfied.
\end{remark}

\begin{remark}
One may ponder whether condition (\ref{condTh2}) always holds for the long
memory case. To settle such concerns, we offer the following counterexample.
Take a linear random field of the form (\ref{deflin}) with $d=1$--i.e., a
linear process. In particular, consider the one-sided linear process with
alternating harmonic coefficients. That is, put $a_{i}=(-1)^{i+1}/i$ for
$i\in{\mathbb{N}}$, and $a_{i}=0$ for $i\in{\mathbb{Z}}\backslash{\mathbb{N}}%
$. In this example, we take $J_{n}=1$ and the index set $\Gamma_{n}^{1}$ to be
the set $\{1,2,\cdots,n\}$. Note that $\sup_{i\in{\mathbb{Z}}}|b_{n,i}|$ does
not go to infinity as $n\rightarrow\infty$, and therefore, the aforementioned
condition is not satisfied. Even though the local limit theorem is not
guaranteed by our Theorem \ref{thm2} for this case, we note that the central
limit theorem holds, since $B_{n}\rightarrow\infty$ as $n\rightarrow\infty$.
\end{remark}

\begin{remark}
In Theorem \ref{thm2} we provide a local limit theorem for linear random
fields when the coefficients are absolutely summable with no restriction on
the sequence of regions other than $B_{n}\rightarrow\infty$. We also provide a
local limit theorem for the sum of a long memory linear random field over a
sequence of regions $\Gamma_{n}^{d}$ which are a disjoint union of discrete
rectangles and with no other specification on the individual coefficients
$a_{i},i\in\mathbb{Z}^{d}$ besides the global conditions (\ref{cond coef}%
)\ and (\ref{condTh2}). In practical application it allows us to have disjoint
discrete rectangles as spatial sampling regions, and the number of these
disjoint spatial rectangular sampling regions may increase as the sample size
increases. The discrete spatial rectangular sampling regions also include
$(\prod_{k=1}^{d}[\underline{n}_{k},\overline{n}_{k}])\cap{\mathbb{Z}}^{d}$
where $\underline{n}_{k}=\overline{n}_{k}$ for some $k$'s. We may have a
single point region if the equality holds for all $k$'s. We would also like to
mention that our local limit results are new also for $d=1$. Furthermore, we
have the freedom to take the sum over $J_{n}$ blocks of random variables as
long as $J_{n}=o(B_{n}^{2})$ for central limit theorem and $J_{n}=o(\sup
_{i\in{\mathbb{Z}^{d}}}|b_{n,i}|/(\log(B_{n}))^{d/2})$ for local limit theorem.
\end{remark}

\begin{remark}
El Machkouri et al. (2013) worked with nonlinear random fields, and in
their work on central limit theorem, they required the condition that
$|\partial\Gamma_{n}^{d}|/|\Gamma_{n}^{d}|\rightarrow0$ as $n\rightarrow
\infty$, where $\partial\Gamma_{n}^{d}$ is the boundary of the region
$\Gamma_{n}^{d}$. We would just like to mention that our results demonstrate
that this condition is not necessary in the linear random field setting. For
example, in the case $d=2$ with $\Gamma_{n}^{2}=([1,n]\times\lbrack
1,3])\cap{\mathbb{Z}}^{2}$, we have $|\partial\Gamma_{n}^{2}|/|\Gamma_{n}%
^{2}|=(2n+2)/3n$.
\end{remark}

\section{Examples}

There are many situations of interest when (\ref{condTh2}) holds. In
particular it is satisfied by the fractionally integrated processes which play
an important role for analyzing various models in econometrics. They are a
particular case of linear processes with regularly varying coefficients for
which we provide a few examples. Of course, examples of this type, where the
coefficients are absolutely summable, will certainly satisfy the local theorem
as given in the first part of Theorem \ref{thm2}. In what follows, we shall
discuss only the long memory case.

\bigskip

\textbf{Example 1}. Suppose we work on one rectangle $\Gamma_{n}^{d}%
=\prod_{\ell=1}^{d}[1,n_{\ell}]\cap{\mathbb{Z}}^{d}$, where $n_{\ell}=n_{\ell
}(n)$ is a sequence of natural numbers for each $\ell$. Let $X_{n}$ and
$B_{n}$ be defined as in (\ref{deflin}) and (\ref{B_n}). For $j=(j_{1}%
,j_{2},\cdots,j_{d})$, let $(a_{j})_{j\in\mathbb{Z}^{d}}$ with $\sum
_{j\in{\mathbb{Z}}^{d}}a_{j}^{2}<\infty$ and assume that for some constant
$C$,%
\begin{equation}
a_{j}\geq C\prod_{\ell=1}^{d}(1/|j_{\ell}|)^{\beta_{\ell}}\text{ with }%
\beta_{\ell}>1/2,\text{ }1\leq\ell\leq d. \label{coeff1}%
\end{equation}
Here we take $1/|j_{\ell}|=1$ if $j_{\ell}=0$. Assume that at least one
$\beta_{\ell}$ is strictly smaller than $1$. Then $(a_{j})_{j\in\mathbb{Z}%
^{d}}$ is not absolutely summable and the linear random field has long memory.
Let us assume now that $\beta_{k}<1,$ for all positive integers $k,$ $1\leq
k\leq m$, for some $m$ with $1\leq m\leq d$ and $\beta_{\ell}>1,$ $m+1\leq
\ell\leq d$. Assume that $n_{k}\rightarrow\infty,$ for all $k$, $1\leq k\leq
m$ and, for some $M<\infty,$ $n_{k}\leq M,$ $m+1\leq k\leq d$. Then, starting
from (\ref{coeff1}), by simple analytical manipulations, we have that
\begin{align*}
b_{n,0}  &  =\sum_{j\in\Gamma_{n}^{d}}a_{j}\geq C\sum_{j\in\Gamma_{n}^{d}%
}\prod_{\ell=1}^{d} (1/|j_{\ell}|)^{\beta_{\ell}}=C\prod_{\ell=1}^{d}%
\sum_{1\leq j_{\ell}\leq n_{\ell}}(1/j_{\ell})^{\beta_{\ell}}\\
&  \propto\prod_{\ell=1}^{m}n_{\ell}^{1-\beta_{\ell}}\prod_{\ell=m+1}^{d}%
\sum_{1\leq j_{\ell}\leq n_{\ell}}(1/j_{\ell})^{\beta_{\ell}}\geq C_{1}%
\prod_{\ell=1}^{m}n_{\ell}^{1-\beta_{\ell}}.
\end{align*}
On the other hand, using (\ref{bound B}), $B_{n}^{2}\leq|\Gamma_{n}^{d}%
|^{2}\mathbb{E}(X_{0}^{2})=\prod_{\ell=1}^{d}n_{\ell}^{2}\mathbb{E}(X_{0}%
^{2})$. Based on these computations, we obtain that
\begin{align*}
\frac{\log B_{n}}{\sup_{\ell\in{\mathbb{Z}}^{d}}|b_{n,\ell}|^{2/d}}  &  \leq
C_{2}\frac{\sum_{\ell=1}^{d}\log n_{\ell}}{b_{n,0}^{2/d}}\\
&  \leq C_{3}\frac{\sum_{\ell=1}^{m}\log n_{\ell}}{\prod_{\ell=1}^{m}n_{\ell
}^{2(1-\beta_{\ell})/d}}\rightarrow 0\text{ \ as }\min_{1\leq j\leq m}%
(n_{j})\rightarrow\infty.
\end{align*}
This latter limit, shows that condition (\ref{condTh2}) is satisfied. Hence,
the local limit theorem in Theorem \ref{thm2} holds, provided Cramèr condition
is satisfied.

\bigskip

In the context of this example, note that we can also consider sets of the
form $\Gamma_{n}^{d}=\bigcup\nolimits_{w=1}^{J_{n}}\Gamma_{n}^{d}(w),$ where
$\{\Gamma_{n}^{d}(w):1\leq w\leq J_{n}\}$ are disjoint rectangles $\Gamma
_{n}^{d}(w)=\prod_{\ell=1}^{d}[c_{w},c_{w}+n_{\ell}]\cap{\mathbb{Z}}^{d}$ of
equal size. For simplicity let us take $m=d$. For this case we have $B_{n}%
^{2}\leq|\Gamma_{n}^{d}|^{2}\mathbb{E}(X_{0}^{2})=J_{n}^{2}\prod_{\ell=1}%
^{d}n_{\ell}^{2}\mathbb{E}(X_{0}^{2})$ and $\sup_{\ell\in{\mathbb{Z}}^{d}%
}|b_{n,\ell}|\geq\prod_{\ell=1}^{d}n_{\ell}^{1-\beta_{\ell}}$. Hence,%
\[
\frac{J_{n}^{2/d}\log B_{n}}{\sup_{\ell\in{\mathbb{Z}}^{d}}|b_{n,\ell}|^{2/d}}\leq
C_{4}\frac{J_{n}^{2/d}(\log J_{n}+\sum_{\ell=1}^{d}\log n_{\ell})}{\prod_{\ell
=1}^{d}n_{\ell}^{2(1-\beta_{\ell})/d}},
\]
which converges to $0$ when $\min_{1\leq j\leq d}(n_{j})\rightarrow\infty,$ as
soon as $J_{n}^{2/d}=o(\prod_{\ell=1}^{d}n_{\ell}^{2(1-\beta_{\ell})/d}/\sum_{\ell=1}%
^{d}\log n_{\ell})$. If we impose this condition on $J_{n}$ then we can also
obtain the conclusion of Theorem \ref{thm2} for this situation.

\bigskip

\textbf{Example 2}. This example is a variant of Example 1, with the same
index sets $\Gamma_{n}^{d}$. Take now
\begin{equation}
a_{j}=\prod_{\ell=1}^{d}(1/|j_{\ell}|)^{\alpha_{\ell}}h_{\ell}(|j_{\ell}|),
\label{coeff2}%
\end{equation}
with $\alpha_{\ell}>1/2$ and $h_{\ell}(\cdot)$ are positive slowly varying
functions, $1\leq\ell\leq d$. Again, we let $1/|j_{\ell}|=1$ if $j_{\ell}=0$.
Then $\sum_{\ell\in{\mathbb{Z}}^{d}}a_{\ell}^{2}<\infty$. If $\alpha_{k}<1$
for some $1\leq k\leq d,$ then $\sum_{\ell\in{\mathbb{Z}}^{d}}|a_{\ell
}|=\infty$ and we are in the long memory case. For some $m$ with $1\leq m\leq
d,$ assume now that $1/2<\alpha_{k}<1,$ for all $k,$ $1\leq k\leq m$, and
$\alpha_{\ell}>1,$ $m+1\leq\ell\leq d$. Recall now that, for a positive slowly
varying function $h(x),$ we have that for every $\varepsilon>0,$
$\lim_{x\rightarrow\infty}x^{\varepsilon}h(x)=\infty$ and $\lim_{x\rightarrow
\infty}x^{-\varepsilon}h(x)=0~$(see Seneta, 1976). Then we can find constants
$1/2<\beta_{k}<1,$ for all $k,$ $1\leq k\leq m$, and $\beta_{\ell}>1,$
$m\leq\ell\leq d$ such that (\ref{coeff1}) holds. If we assume that
$n_{k}\rightarrow\infty,$ for all $k$, $1\leq k\leq m$ and, for some
$M<\infty,$ $n_{k}<M,$ $m+1\leq k\leq d$, then the conditions in Example 1
hold. Therefore, for this case, the conclusion of Theorem \ref{thm2} holds
with
\[
B_{n}^{2}=\prod_{\ell=1}^{m}c(\alpha_{\ell})n_{\ell}^{3-2\alpha_{\ell}}%
h_{\ell}^{2}(n_{\ell}),
\]
with constants $c(\alpha_{\ell})$ specified in Wang, Lin and Gulati (2001).\bigskip

\textbf{Example 3. } We work this time on one rectangle $\Gamma_{n}^{d}%
=\prod_{\ell=1}^{d}[1,k_{i}n]\cap{\mathbb{Z}}^{d}$, where $k_{i}\in
{\mathbb{R}}^{+},1\leq i\leq d$. For $j=(j_{1},j_{2},\cdots,j_{d})$, let
$(a_{j})_{j\in\mathbb{Z}^{d}}$ with $\sum_{\ell\in{\mathbb{Z}}^{d}}a_{\ell
}^{2}<\infty$ and assume that for some constant $C>0$,
\begin{equation}
a_{j}\geq C||j||^{-\beta}\text{ with }\beta\in(d/2,d)\text{ and }j\neq
0_{d}=(0,0,...,0).\label{lowerB2}%
\end{equation}
It is easy the see that $\sum_{\ell\in{\mathbb{Z}}^{d}}a_{\ell}=\infty$ and we
also have $a_{j}\geq C(j_{1}+j_{2}+\cdots+j_{d})^{-\beta}.$ Straightforward
computations show that we can find a positive constant $C_{1\text{ }}$ and
$n_{0}\in\mathbb{N}$ such that for all $n\ >n_{0}$ we have%
\[
b_{n,0}=\sum_{j\in\Gamma_{n}^{d}}a_{j}\geq C_{1}n^{d-\beta}\text{ }.
\]
Therefore, using (\ref{bound B}), we can find a positive constant $C_{2}$ such
that
\begin{align*}
\frac{\log(B_{n})}{\sup\limits_{i\in{\mathbb{Z}}^{d}}b_{n,i}^{2/d}} &  \leq
\frac{C_{2}\log(n^{d})}{b_{n,0}^{2/d}}\\
&  \leq\frac{C_{2}d\log{n}}{C_{1}n^{2(d-\beta)/d}}\rightarrow0\text{ as
}n\rightarrow\infty.
\end{align*}
This shows that condition (\ref{condTh2}) of Theorem \ref{thm2} is satisfied
and the local limit theorem holds if Cramèr condition is satisfied.

\bigskip

Furthermore, we can also mention that for this case the local limit theorem
also holds if we actually consider an union of $J_{n}$ rectangles of equal
size $(k_{1} n,...,k_{d} n)$ such that $J_{n}^{2/d}=o(n^{2(d-\beta)/d}/\log n)$.

\bigskip

As a particular example of this kind we shall give an example treated by
Surgailis (1982) and also by Beknazaryan et al. (2019).

\bigskip

\textbf{Example 4}. Assume that $\Gamma_{n}^{d}$ are cubic, i.e., $\Gamma
_{n}^{d}=[-n,n]^{d}\cap{\mathbb{Z}}^{d}$, and put $a_{i}=l(\Vert
i\Vert)G(i/\Vert i\Vert)\Vert i\Vert^{-\alpha}$ with $\alpha\in(d/2,d),$
where
$l(x)$ is slowly varying at $\infty$ and $G:\mathbb{S}_{d-1}\rightarrow
{\mathbb{R}}^{+}$ is continuous on its domain \textup{\big(}the unit sphere in
$d$-dimensional space\textup{\big)}. For this example we know that
$B_{n}\propto n^{\frac{3d}{2}-\alpha}\;l(n)$ (see Surgailis, 1982,
Theorem 2) and from Beknazaryan et al. (2019) we can easily deduce
that $\sup\nolimits_{i\in{\mathbb{Z}}^{d}}\big|b_{n,i}\big|\propto
(n^{d-\alpha}\;l(n))$. We could also see directly that condition
(\ref{condTh2}) of Theorem \ref{thm2} is satisfied\ by using the proof of
Example 3. Indeed, by the properties of slowly varying functions, we can find
$\beta\in(d/2,d)$ such that $a_{j}\geq C||j||^{-\beta}.$ Since we are in the
long memory case, if the innovations satisfy the Cramér condition, then
(\ref{LCLTnonlat}) holds. \rule{0.5em}{0.5em}

\bigskip

\section{Simulation Study}

In this section, we perform a simulation study for the local limit theorem in
Example $4,$ applied to the one-dimensional case. The linear processes we used
here are the fractionally integrated processes FARIMA$(0,1-\alpha,0)$ which
play an important role in financial time series modeling, and they are widely
studied. Such processes are defined for $1/2<\alpha<1$ by
\[
X_{j}=(1-B)^{\alpha-1}\varepsilon_{j}=\sum_{i\geq0}a_{i}\varepsilon
_{j-i}\text{ with }a_{i}=\frac{\Gamma(i+1-\alpha)}{\Gamma(1-\alpha
)\Gamma(i+1)}\,\text{\ ,}\label{deffractlin}%
\]
where $B$ is the backward shift operator, $B\varepsilon_{j}=\varepsilon_{j-1}%
$. By the well-known fact that $\lim_{n\rightarrow\infty}\Gamma(n+x)/n^{x}%
\Gamma(n)=1$ for any real $x,$ we have $\lim_{n\rightarrow\infty}%
a_{n}/n^{-\alpha}=1/\Gamma(1-\alpha)$. The variance of the partial sum
$S_{n}=\sum_{j=1}^{n}X_{j}$ is
\begin{equation}
B_{n}^{2}\sim c_{\alpha}n^{3-2\alpha}{{\mathbb{E}}\,}\varepsilon
^{2}/[(1-\alpha)(3-2\alpha)\Gamma^{2}(1-\alpha)] \label{defBn}%
\end{equation}
where
\[
c_{\alpha}=\int_{0}^{\infty}x^{-\alpha}(1+x)^{-\alpha}dx.
\]
The variance formula for the partial sum of FARIMA$(0,1-\alpha,0)$ is well
known. See, for example, Wang, Lin and Gulati (2001).

Using the FARIMA$(0,1-\alpha,0)$ model, linear processes with innovations
following the Student's $t$ distribution with 5 degrees of freedom were
generated. Employing the MATLAB code of Fay et al. (2009), $N$ replicates
of linear processes were generated, each of length $n$. Specifically, we
generated cases with $N=5,000$ and $N=10,000$ cross-referenced with $n=2^{10}%
$, $n=2^{12}$, and $n=2^{14}$, and this was done for each of the values
$\alpha=0.95$, $\alpha=0.70$, and $\alpha=0.55$. Once the data were obtained,
the local limit measures of various intervals were estimated by using relative
frequency to estimate $P(S_{n}\in(a,b))$ and using the approximation of
$B_{n}$ given in (\ref{defBn}).

The simulation study supports the validity of Example 4 for the
one-dimensional case. See Tables \ref{table1} and \ref{table2} below. Of particular interest is the general tendency of
results to be better for larger $N$, which is likely explained by the fact
that we estimate $P(S_{n}\in(a,b))$ using relative frequency. Also, we notice
that the results generally get better with larger values of the sample size
$n$.



\begin{table}[H]
\caption{Local limit measures of the intervals $(-100,0),(-50,50)$, and
$(0,100)$ - one per row - using $N$ one-dimensional linear processes, each of
length $n$, employing various long memory cases using the $FARIMA(0,\;1-\alpha
\;,0)$ model with $t_{5}$ innovations.}%
\label{table1}
\center
\bigskip
{\scriptsize \
\begin{tabular}
[c]{|c|ccc|ccc|ccc|}\hline
&  &  &  &  &  &  &  &  & \\
& \multicolumn{3}{|c|}{$n=2^{10}$} & \multicolumn{3}{|c|}{$n=2^{12}$} &
\multicolumn{3}{|c|}{$n=2^{14}$}\\
&  &  &  &  &  &  &  &  & \\
$N$ & $\alpha=0.95$ & $\alpha=0.70$ & $\alpha=0.55$ & $\alpha=0.95$ &
$\alpha=0.70$ & $\alpha=0.55$ & $\alpha=0.95$ & $\alpha=0.70$ & $\alpha
=0.55$\\\hline
& $66$ & $105$ & $117$ & $92$ & $99$ & $98$ & $95$ & $91$ & $122$\\
$5\times10^{3}$ & $90$ & $99$ & $115$ & $101$ & $95$ & $108$ & $100$ & $88$ &
$110$\\
& $67$ & $99$ & $97$ & $90$ & $96$ & $108$ & $98$ & $106$ & $110$\\\hline
& $67$ & $97$ & $105$ & $91$ & $98$ & $101$ & $96$ & $97$ & $104$\\
$1\times10^{4}$ & $89$ & $98$ & $95$ & $99$ & $103$ & $105$ & $101$ & $97$ &
$98$\\
& $65$ & $103$ & $101$ & $87$ & $104$ & $108$ & $98$ & $98$ & $92$\\\hline
\end{tabular}
}\end{table}
\begin{table}[H]
\caption{Local limit measures of the intervals $(-50,0),(-25,25)$, and
$(0,50)$ - one per row - using $N$ one-dimensional linear processes, each of
length $n$, employing various long memory cases using the $FARIMA(0,\;1-\alpha
\;,0)$ model with $t_{5}$ innovations.}
\center
\bigskip
{\scriptsize \
\label{table2}
\begin{tabular}
[c]{|c|ccc|ccc|ccc|}\hline
&  &  &  &  &  &  &  &  & \\
& \multicolumn{3}{|c|}{$n=2^{10}$} & \multicolumn{3}{|c|}{$n=2^{12}$} &
\multicolumn{3}{|c|}{$n=2^{14}$}\\
&  &  &  &  &  &  &  &  & \\
$N$ & $\alpha=0.95$ & $\alpha=0.70$ & $\alpha=0.55$ & $\alpha=0.95$ &
$\alpha=0.70$ & $\alpha=0.55$ & $\alpha=0.95$ & $\alpha=0.70$ & $\alpha
=0.55$\\\hline
& $46$ & $51$ & $67$ & $51$ & $52$ & $62$ & $49$ & $45$ & $61$\\
$5\times10^{3}$ & $50$ & $47$ & $54$ & $49$ & $49$ & $43$ & $48$ & $40$ & $61
$\\
& $46$ & $48$ & $48$ & $49$ & $43$ & $46$ & $51$ & $43$ & $49$\\\hline
& $46$ & $48$ & $50$ & $49$ & $52$ & $43$ & $50$ & $51$ & $55$\\
$1\times10^{4}$ & $48$ & $50$ & $51$ & $50$ & $52$ & $44$ & $50$ & $45$ & $49
$\\
& $43$ & $51$ & $54$ & $50$ & $51$ & $62$ & $51$ & $47$ & $43$\\\hline
\end{tabular}
}\end{table}


\section{Proofs}

For the proof of Theorem \ref{thm2}, we need several lemmas.

\begin{lemma}
\label{lemma1} Let $\varphi(t)$ be the characteristic function of some random
variable, and let $b\text{ and }c<1$ be positive real numbers. If
$|\varphi(t)|\leq c\text{ for }b\leq|t|\leq2b$, then
\[
|\varphi(t)|\leq1-\frac{1-c^{2}}{8b^{2}}t^{2}\text{ \ for all }|t|<b.
\]

\end{lemma}

\begin{proof}
This is a version of Theorem 1 on page 10 in Petrov (1975), which is
obtained by using the same proof.
\end{proof}

\begin{lemma}
\label{lemma2} If $\varphi(t)$ is the characteristic function of some random
variable satisfying the Cramér condition, then for any $\delta>0$ there is
$\beta=\beta(\delta)\in(0,1)$ such that
\[
|\varphi(t)|\leq\beta\text{ for all }|t|\geq\delta.
\]

\end{lemma}

\begin{proof}
Since $\lim\sup_{|t|\rightarrow\infty}|\varphi(t)|<1$, there exists
$0<\gamma<1$ and $T>0$ such that for all $|t|>T$ we have that $|\varphi
(t)|\leq\gamma.$ For any $\delta>0$ such that $\delta> T$, the result holds
with $\beta=\gamma$. By Lemma \ref{lemma1}, on the other hand, $\lim
\sup_{|t|\rightarrow\infty}|\varphi(t)|<1$ implies that $|\varphi(t)|<1$ for
all $t\neq0$. If $\delta<T$, we appeal to the continuity of $\varphi(t)$ to
guarantee that $\eta=\max_{\delta\le|t|\le T}|\varphi(t)|\in(0,1)$, whence
$|\varphi(t)|\leq\eta$ for any $t$ with $|t|\in\lbrack\delta,T]$. Therefore,
the result holds with $\beta=\gamma\vee\eta.$
\end{proof}

\begin{lemma}
\label{lemma3} If $S_{n}$ and $B_{n}$ are as defined in (\ref{S_n}) and
(\ref{B_n}) respectively, if the innovations have a non-lattice distribution,
and if $\{\omega_{n}\}$ is a sequence of positive real numbers for which there
exists some $M>0$ so that $|\omega_{n}b_{n,i}|\leq M$ for all $n\in
{\mathbb{N}}$ and for all $i\in{\mathbb{Z}}^{d}$, then the function
\[
\varphi_{\frac{S_{n}}{B_{n}}}(t)\ \mathcal{I}(|t|<\omega_{n}B_{n})
\]
is dominated by some integrable function $g(t)$.
\end{lemma}

\begin{proof}
Since we assume that $|\varphi_{\varepsilon}(t)|<1$ for all $t\neq0$, because
$\varphi_{\varepsilon}(t)$ is continuous, there exists $c=c(M)\in(0,1)$ such
that $|\varphi_{\varepsilon}(u)|\leq c$ for $M\leq|u|\leq2M.$ By Lemma
\ref{lemma1} and because of the inequality $1-x\leq e^{-x}$ for all
$x\in{\mathbb{R}}$, we deduce that $|u|\leq M$ implies that
\[
\text{ }|\varphi_{\varepsilon}(u)|\leq1-\frac{1-c^{2}}{8M^{2}}u^{2},
\]
and therefore
\[
|\varphi_{\varepsilon}(u)|\leq\exp\bigg(-\frac{1-c^{2}}{8M^{2}}u^{2}\bigg).
\]
Now, by independence, we have
\[
\varphi_{\frac{S_{n}}{B_{n}}}(t)=\varphi_{\sum\limits_{i\in{\mathbb{Z}}^{d}%
}b_{n,i}\varepsilon_{i}}\bigg(\frac{t}{B_{n}}\bigg)=\prod\limits_{i\in
{\mathbb{Z}}^{d}}\varphi_{\varepsilon}\bigg(\frac{b_{n,i}\ t}{B_{n}}\bigg).
\]
For $|t|<\omega_{n}B_{n}$, we observe that
\[
\bigg|\frac{b_{n,i}\ t}{B_{n}}\bigg|<|b_{n,i}\ \omega_{n}|\le M.
\]
Overall, we have
\begin{align*}
\big|\varphi_{\frac{S_{n}}{B_{n}}}(t)\big|\ \mathcal{I}(|t|<\omega_{n}B_{n})
&  =\prod_{i\in{\mathbb{Z}}^{d}}\bigg|\varphi_{\varepsilon}\bigg(\frac
{b_{n,i}\ t}{B_{n}}\bigg)\bigg|\ \mathcal{I}(|t|<\omega_{n}B_{n})\\
&  \leq\prod_{i\in{\mathbb{Z}}^{d}}\exp\bigg(-\frac{1-c^{2}}{8M^{2}}%
\frac{b_{n,i}^{2}}{B_{n}^{2}}t^{2}\bigg)\\
&  =\exp\bigg(-\frac{1-c^{2}}{8M^{2}}\frac{1}{\sigma_{\varepsilon}^{2}}%
t^{2}\bigg),
\end{align*}
which we take to be our desired dominating integrable function $g(t)$.
\end{proof}

For use in the following lemma, we shall introduce the following notation. For
a countable collection of real numbers $\{b_{j}:j\in{\mathbb{Z}}^{d}\}$, where
$j=(j_{1},...,j_{d})$, we denote an increment in the direction $k$ by
\[
\Delta_{k}b_{j_{1},...,j_{k},...,j_{d}}=b_{j_{1},...,j_{k},...,j_{d}}%
-b_{j_{1},...,j_{k}-1,...,j_{d}}%
\]
and their composition is denoted by $\Delta:$
\begin{equation}
\Delta b_{j}=\Delta_{1}\circ\Delta_{2}\circ...\circ\Delta_{d}b_{j}.
\label{Incr}%
\end{equation}
For instance, if $d=1$ we have $\Delta b_{j}=b_{j}-b_{j-1}.$ For $d=2$,%
\begin{equation}
\Delta b_{i,j}=\Delta_{1}\circ\Delta_{2}b_{i,j}=b_{i,j}-b_{i,j-1}%
-b_{i-1,j}+b_{i-1,j-1}. \label{incr2}%
\end{equation}
Denote $\sum\nolimits_{i\in{\mathbb{Z}}^{d}}a_{i}^{2}=D^{2}<\infty.$ Define as
before $b_{i}=b_{n,i}=b_{i}(n)=\sum\nolimits_{j\in\Gamma_{n}^{d}}a_{j-i}$. For
$k\in{\mathbb{N}},$ and for $j\in{\mathbb{Z}}^{d}$ we denote by $V_{k}(j)$ the
vertices of the cube $\prod\nolimits_{1\leq\ell\leq d}[j_{\ell}-k,j_{\ell}].$

For the proof of the long memory case in Theorem \ref{thm2}, we need the
following lemma about the size of the coefficients $b_{i}=b_{n,i}.$

\begin{lemma}
\label{Lmbn} For any $\ell\in{\mathbb{Z}}^{d}$ and any $k\ge1$ we have
\[
\sum\nolimits_{u\in V_{k}(\ell),u\neq\ell}|b_{u}|\geq\ |b_{\ell}|-2^{d}%
DJ_{n}k^{\frac{d}{2}}\ .
\]

\end{lemma}


\begin{proof}
In order to avoid complicated notation, we shall (without loss of generality)
prove in detail the case $d=2$, and the general case will follow by a similar
argument. For this case, the increment $\Delta$ is defined by (\ref{incr2}),
$b_{i,j}=\sum\nolimits_{(s,t)\in\Gamma_{n}^{2}}a_{s-i,t-j}$, where $\Gamma
_{n}^{2}$ is as defined in (\ref{defgamma}). Since
\[
\sum_{u=i-k+1}^{i}\;\sum_{v=j-k+1}^{j}\Delta b_{u,v}=b_{i,j}-b_{i,j-k}%
-b_{i-k,j}+b_{i-k,j-k},
\]
we employ the triangle inequality to obtain
\begin{equation}
|b_{i,j}|\leq|b_{i,j-k}|+|b_{i-k,j}|+|b_{i-k,j-k}|+\sum\nolimits_{u=i-k+1}%
^{i}\sum\nolimits_{v=j-k+1}^{j}|\Delta b_{u,v}|.\label{ineqbij}%
\end{equation}
By the linearity of $\Delta$ and the definition of $\Gamma_{n}^{2}$, we notice
that
\[
\Delta b_{u,v}=\Delta\bigg[\sum_{(s,t)\in\Gamma_{n}^{2}}a_{s-u,t-v}%
\bigg]=\Delta\bigg[\sum_{w=1}^{J_{n}}\sum_{(s,t)\in\Gamma_{n}^{2}%
(w)}a_{s-u,t-v}\bigg]=\sum_{w=1}^{J_{n}}\Delta\bigg[\sum_{(s,t)\in\Gamma
_{n}^{2}(w)}a_{s-u,t-v}\bigg].
\]
For fixed $w\in\{1,2,...,J_{n}\}$, let us investigate the expression
$\Delta\bigg[\sum\limits_{(s,t)\in\Gamma_{n}^{2}(w)}a_{s-u,t-v}\bigg]$.
Indeed, after some cancellations, we get
\begin{align*}
\Delta\bigg[ &  \sum\limits_{(s,t)\in\Gamma_{n}^{2}(w)}a_{s-u,t-v}\bigg]\\
&  =\sum_{(s,t)\in\Gamma_{n}^{2}(w)}\big[a_{s-u,t-v}-a_{s-u,t-(v-1)}%
-a_{s-(u-1),t-v}+a_{s-(u-1),t-(v-1)}\big]\\
&  =\sum_{s=\underline{n}_{1}(\omega)}^{\overline{n}_{1}(\omega)}%
\sum_{t=\underline{n}_{2}(\omega)}^{\overline{n}_{2}(\omega)}\big[a_{s-u,t-v}%
-a_{s-u,(t+1)-v}-a_{(s+1)-u,t-v}+a_{(s+1)-u,(t+1)-v}\big]\\
&  =\sum_{s=\underline{n}_{1}(\omega)}^{\overline{n}_{1}(\omega)}%
\big[a_{s-u,\underline{n}_{2}(\omega)-v}-a_{s-u,(\overline{n}_{2}%
(\omega)+1)-v}-a_{(s+1)-u,\underline{n}_{2}(\omega)-v}+a_{(s+1)-u,(\overline
{n}_{2}(\omega)+1)-v}\big]\\
&  =a_{\underline{n}_{1}(w)-u,\underline{n}_{2}(w)-v}-a_{(\overline{n}%
_{1}(w)+1)-u,\underline{n}_{2}(w)-v}+a_{(\overline{n}_{1}(w)+1)-u,(\overline
{n}_{2}(w)+1)-v}-a_{\underline{n}_{1}(w)-u,(\overline{n}_{2}(w)+1)-v}.
\end{align*}
This identity together with the Cauchy-Schwarz inequality, demonstrate that
\[
\sum\nolimits_{u=i-k+1}^{i}\sum\nolimits_{v=j-k+1}^{j}|\Delta b_{u,v}%
|\leq(4Dk)J_{n}.
\]
Therefore, combining this latter inequality with (\ref{ineqbij}), we obtain
\[
|b_{i,j-k}|+|b_{i-k,j}|+|b_{i-k,j-k}|\geq|b_{i,j}|-4DJ_{n}k,
\]
thereby establishing the result for $d=2$. For general $d$, the difference is
that we use the formula (\ref{Incr}) instead of (\ref{incr2}) and we take into
account that, in this case, the number of vertices of the cube $\prod
\nolimits_{1\leq\ell\leq d}[j_{\ell}-k,j_{\ell}]$ is $2^{d}.$
\end{proof}

\begin{lemma}
\label{gamma} Assume that conditions of Theorem \ref{thm2} are satisfied for
the long memory case. Then there exists $0<\rho<1$ independent of $n$, such
that for all $n$ sufficiently large
\[
|\varphi_{S_{n}}(t)|\mathcal{I}(|t|\geq\gamma_{n}^{-1})\leq\rho^{\lfloor
\left(\gamma_{n}/J_{n}\right)^{2/d}\rfloor},
\]
where $\gamma_{n}=\sup_{i\in{\mathbb{Z}}^{d}}|b_{n,i}|.$
\end{lemma}

\begin{proof}
To simplify the notation, we will drop the index $n$ and simply write $b_{i}$
in place of $b_{n,i}$. Assume $|b_{j_{0}}|=\sup_{i\in{\mathbb{Z}}^{d}}%
|b_{i}|.$ Such a $j_{0}$ exists, because $\sum\nolimits_{i\in{\mathbb{Z}}^{d}%
}a_{i}^{2}=D^{2}<\infty.$ Fix $\alpha\in(0,1)$, and denote by $k_{0}$ the
integer part of $[(1-\alpha)|b_{j_{0}}|/(2^{d}DJ_{n})]^{2/d}$, namely
\begin{equation}
k_{0}=\bigg\lfloor\left(  \frac{(1-\alpha)|b_{j_{0}}|}{2^{d}DJ_{n}}\right)
^{\frac{2}{d}}\bigg\rfloor. \label{defk0}%
\end{equation}
Since $B_{n}\rightarrow\infty$, condition (\ref{condTh2}) in Theorem
\ref{thm2} implies that $|b_{j_{0}}|\rightarrow\infty$ and $J_{n}=o(\sup
_{i\in\mathbb{Z}^{d}}|b_{n,i}|)$ as $n\rightarrow\infty$. Therefore the
$k_{0}$ in (\ref{defk0}) satisfies $k_{0}\rightarrow\infty$ as $n\rightarrow
\infty$. So, for $n$ sufficiently large, $k_{0}\geq1$.

By Lemma \ref{Lmbn}, for $1\leq k\leq k_{0}$,
\[
\sum\nolimits_{u\in V_{k}(j_{0}),u\neq j_{0}}|b_{u}|\geq|b_{j_{0}}%
|-2^{d}DJ_{n}k^{\frac{d}{2}}\geq\alpha|b_{j_{0}}|,
\]
which immediately gives
\[
\max_{u\in V_{k}(j_{0}),u\neq j_{0}}|b_{u}|\geq\alpha|b_{j_{0}}|/(2^{d}-1).
\]
So, on the set $|t|\geq\gamma_{n}^{-1}=|b_{j_{0}}|^{-1}$
\begin{equation}
\max_{u\in V_{k}(j_{0}),u\neq j_{0}}|tb_{u}|\geq\alpha/(2^{d}-1). \label{maxb}%
\end{equation}
With these preliminaries in place, let us define
\[
\eta_{k}=\sum\nolimits_{u\in V_{k}(j_{0}),u\neq j_{0}}b_{u}\varepsilon_{u}%
\]
and for $u\in V_{k}(j_{0}),u\neq j_{0}$
\[
\varphi_{u}(t)=\varphi_{\varepsilon}(tb_{u}).
\]
By independence
\[
|{{\mathbb{E}}\,}(\exp(it\eta_{k}))|=\Pi_{u}|\varphi_{u}(t)|\leq
\min\nolimits_{u}|\varphi_{u}(t)|,
\]
where the product and the minimum are over $u\in V_{k}(j_{0}),u\neq j_{0}$.
Since the characteristic function of $\varepsilon_{0}$ satisfies the Cramér
condition, by Lemma \ref{lemma2}, for $\alpha/(2^{d}-1)>0$ we can find
$0<\beta=\beta(\alpha)<1$ such that $|\varphi_{\varepsilon}(s)|\leq\beta$ for
all $|s|\geq\alpha/(2^{d}-1).$ As a consequence, for $k\leq k_{0},$ by
(\ref{maxb}), at least one of $|\varphi_{u}(t)|$ is smaller than $\beta,$
i.e.,
\[
\min\nolimits_{u}|\varphi_{u}(t)|\leq\beta\text{ for all }|t|\geq|b_{j_{0}%
}|^{-1}\text{ and }k\leq k_{0}.
\]
It implies that
\begin{equation}
|{{\mathbb{E}}\,}(\exp(it\eta_{k}))|\leq\beta<1\text{ for all }|t|\geq
|b_{j_{0}}|^{-1}\text{ and }k\leq k_{0}. \label{ineqeta}%
\end{equation}
And since
\[
|\varphi_{S_{n}}(t)|\mathcal{=}\prod_{j\in{\mathbb{Z}}^{d}}|\varphi
_{\varepsilon}(b_{j}t)|\text{ }\leq\prod_{1\leq k\leq k_{0}}|{{\mathbb{E}}%
\,}(\exp(it\eta_{k}))|,
\]
by inequality (\ref{ineqeta}) and the definition of $k_{0}$ in (\ref{defk0}),\ it follows that
\[
|\varphi_{S_{n}}(t)|\mathcal{I}(|t|\geq\gamma_{n}^{-1})\leq\text{ }%
\beta^{k_{0}}\le \rho^{\lfloor\left(\gamma_{n}/J_{n}\right)^{2/d}\rfloor}%
\]
for some $\rho\in(0,1)$.
\end{proof}

\paragraph{Proof of Theorem \ref{thm2}}

The proof is based, as usual, on the study of the characteristic function of
the sum $S_{N.}$ As in Hafouta and Kifer (2016),
we prove (\ref{LCLTnonlat}) for all continuous complex-valued functions $h$ defined on
${\mathbb{R}}$, $|h|\in L^{1}({\mathbb{R}})$ such that%

\[
\hat{h}(t)=\int\nolimits_{{\mathbb{R}}}e^{-itx}h(x)dx
\]
is real-valued and has compact support contained in some finite interval
$[-L,L]$. By the inversion formula%
\[
h(t)=\frac{1}{2\pi}\int\nolimits_{{\mathbb{R}}}e^{itx}\hat{h}(x)\ dx.
\]
Employing a change of variables, we see that%

\begin{equation}
{{\mathbb{E}}\,}[h(S_{n}-u)]=\frac{1}{2\pi B_{n}}\int_{{\mathbb{R}}}\hat
{h}\bigg(\frac{t}{B_{n}}\bigg)\;\varphi_{\frac{S_{n}}{B_{n}}}(t)\;\exp
\bigg(-\frac{itu}{B_{n}}\bigg)\ dt. \label{18}%
\end{equation}
By the Fourier inversion formula we also have
\begin{equation}
\exp\bigg(-\frac{u^{2}}{2B_{n}^{2}}\bigg)=\frac{1}{\sqrt{2\pi}}\int
_{{\mathbb{R}}}\exp\bigg(-\frac{itu}{B_{n}}\bigg)\exp\bigg(-\frac{t^{2}}%
{2}\bigg)\ dt. \label{18b}%
\end{equation}
By (\ref{18}) and (\ref{18b})\ and some simple algebraic manipulations, we obtain%

\begin{align}
&  \sup_{u\in\mathbb{R}}\bigg |\sqrt{2\pi}\;B_{n}\;E[h(S_{n}-u)]-\exp
\bigg(-\frac{u^{2}}{2B_{n}^{2}}\bigg)\int_{{\mathbb{R}}}%
h(x)dx\bigg|\;\;\;\;\;\;\;\;\;\;\;\;\label{18c}\\
& \nonumber\\
&  \;\;\;\leq\frac{1}{\sqrt{2\pi}}\int_{{\mathbb{R}}}\bigg|\hat{h}%
\bigg(\frac{t}{B_{n}}\bigg)\;\varphi_{\frac{S_{n}}{B_{n}}}(t)-\exp
\bigg(-\frac{t^{2}}{2}\bigg)\int_{{\mathbb{R}}}h(x)dx\bigg|\ dt.\nonumber\\
& \nonumber
\end{align}
As in Lemma \ref{gamma} denote $\gamma_{n}=\sup_{i\in{\mathbb{Z}}^{d}}%
|b_{n,i}|.$ At this point, note that because we have $\gamma_{n}\leq B_{n}$,
condition (\ref{condTh2})\ implies that $J_{n}=o(B_{n}^{2})$ in the long
memory case. By the CLT in Theorem \ref{thm1}, for all $T>0,$ it follows that%
\[
\int_{|t|\leq T}\bigg|\varphi_{\frac{S_{n}}{B_{n}}}(t)-\exp\bigg(-\frac{t^{2}%
}{2}\bigg)\bigg|\ dt\rightarrow0\text{ as }n\rightarrow\infty.
\]
On the other hand
\[
\int_{|t|\geq T}\exp\bigg(-\frac{t^{2}}{2}\bigg)\ dt\rightarrow0\text{ as
}T\rightarrow\infty.
\]
Since $h$ is integrable, $\hat{h}$ is continuous, and $B_{n}\rightarrow\infty
$, for all $t$
\[
\lim_{n\rightarrow\infty}\hat{h}\bigg(\frac{t}{B_{n}}\bigg)=\int_{{\mathbb{R}%
}}h(x)\ dx.
\]
Combining these facts with (\ref{18c}), we note that, in order to obtain the
conclusion of Theorem \ref{thm2}, it suffices to show that%
\begin{equation}
\lim_{T\rightarrow\infty}\lim\sup_{n}\int_{T\leq|t|\leq LB_{n}}\big|\varphi
_{\frac{S_{n}}{B_{n}}}(t)\big|\ dt=0. \label{to verify}%
\end{equation}

First, we deal with the situation when coefficients are absolutely summable.
Since $Lb_{n,i}$ is uniformly bounded by $L\sum\nolimits_{i\in{\mathbb{Z}}%
^{d}}|a_{i}|$, Lemma \ref{lemma3}, applied with $\omega_{n}=L$ and
$M=L\sum\nolimits_{i\in{\mathbb{Z}}^{d}}|a_{i}|,$ guarantees that the
integrand of (\ref{to verify}) is dominated by some integrable function. In
order to verify (\ref{to verify})\ we have just to apply the Lebesgue
dominated convergence theorem.

Henceforth, we confine our attention to the long memory case. We decompose the
region of integration in (\ref{to verify}), yielding%
\begin{align*}
\int_{T\leq|t|\leq LB_{n}}\big|\varphi_{\frac{S_{n}}{B_{n}}}(t)\big|\ dt  &
\leq\int_{T\leq|t|\leq\gamma_{n}^{-1}B_{n}}\big|\varphi_{\frac{S_{n}}{B_{n}}%
}(t)\big|\ dt+\int_{\gamma_{n}^{-1}B_{n}\leq|t|\leq LB_{n}}\big|\varphi
_{\frac{S_{n} }{B_{n}}}(t)\big|\ dt\\
&  =:I_{1, n}+I_{2, n},
\end{align*}
so that we may deal with $I_{1, n}$ and $I_{2, n}$ separately. In what
follows, our objective is to show that both $I_{1, n}\rightarrow0$ and $I_{2,
n}\rightarrow0$ as $n\rightarrow\infty$.

Since $\gamma_{n}^{-1}b_{n,i}$ is uniformly bounded by one, Lemma \ref{lemma3}
applied with $\omega_{n}=\gamma_{n}^{-1}$ and $M=1\ $guarantees that the
integrand of $I_{1, n}$ is dominated by some integrable function $g(t)$. Ergo,
by the Lebesgue dominated convergence theorem, we have
\[
\lim\sup_{n}I_{1, n}\leq\int_{T\leq|t|}g(t)\ dt\rightarrow0\text{ as
}T\rightarrow\infty,
\]
which is exactly what we wished to show about $I_{1, n}$.

Now we proceed to show that $I_{2,n}\rightarrow0$.
By a change of variable, we deduce that
\[
I_{2,n}= B_{n}\;\int_{\gamma_{n}^{-1}<|t|\leq L}\;\big|\varphi_{S_{n}%
}(t)\big|\ dt.
\]
By Lemma \ref{gamma}
\[
|\varphi_{S_{n}}(t)|\mathcal{I}(|t|\geq\gamma_{n}^{-1})\leq\rho^{\lfloor
\left(\gamma_{n}/J_{n}\right)^{2/d}\rfloor},
\]
and so
\[
I_{2,n}\leq B_{n}(2L)\rho^{\lfloor\left(\gamma_{n}/J_{n}\right)^{2/d}\rfloor}.
\]
It is easy to see that $|I_{2,n}|\rightarrow0$ if we impose (\ref{condTh2}).
\rule{0.5em}{0.5em}



\textbf{Acknowledgement} This paper is dedicated to the memory of Murray Rosenblatt. The authors would like to thank the two referees
for their very valuable comments which improved the presentation of this
paper. The research of Magda Peligrad is partially supported by NSF grant
DMS-1811373. The research of Hailin Sang is partially supported by the Simons
Foundation Grant 586789.\\

\textbf{Data Availability Statement}  The data that support the findings of this study are available from the corresponding author upon reasonable request.

\end{document}